\documentclass[12pt]{article}

\usepackage{fullpage}
\usepackage{amssymb}
\usepackage{amsmath}
\usepackage{amsthm}
\usepackage{stmaryrd}
\usepackage{xypic}
\usepackage{setspace}
\usepackage{array}
\usepackage{rotating}

\newtheorem{theorem}{Theorem}[section]
\newtheorem{lemma}[theorem]{Lemma}
\newtheorem{proposition}[theorem]{Proposition}

\newcommand{\tref}[1]{Theorem~\textup{\ref{thm:#1}}}
\newcommand{\pref}[1]{Proposition~\textup{\ref{prop:#1}}}
\newcommand{\cref}[1]{Corollary~\textup{\ref{cor:#1}}}

\newcommand{\lref}[1]{Lemma~\textup{\ref{lem:#1}}}

\newcommand{\Z}{\mathbb Z}

\newcommand{\gp}[2]{\langle\, #1 \,  | \, #2 \,\rangle}
\newcommand{\s}{\sigma}
\newcommand{\G}{\Gamma}

\newcommand{\comment}[1]{}
\newcommand{\eps}{\varepsilon}
    
\newcommand{\calP}{\mathcal P}
\newcommand{\calQ}{\mathcal Q}

\newcommand{\shp}{\!+\!}
\newcommand{\shm}{\!-\!}

\begin{document}

\title{Tight orientably-regular polytopes}

\author{Marston Conder \\
Department of Mathematics\\
University of Auckland\\
Auckland 1142, New Zealand\\
and \\
Gabe Cunningham\\
Department of Mathematics\\
University of Massachusetts Boston\\
Boston, Massachusetts 02125,  USA
}

\date{ \today }
\maketitle

\begin{abstract}

Every equivelar abstract polytope of type $\{p_1, \ldots, p_{n-1}\}$ has at least $2 p_1 \cdots p_{n-1}$ flags.
Polytopes that attain this lower bound are called \emph{tight}. Here we investigate the question of
under what conditions there is a tight orientably-regular polytope of type $\{p_1, \ldots, p_{n-1}\}$.
We show that it is necessary and sufficient that whenever $p_i$ is odd, both $p_{i-1}$ and $p_{i+1}$ are
even divisors of $2p_i$.



\end{abstract}

\section{Introduction}

	Abstract polytopes are ranked partially-ordered sets that resemble the face-lattice of a convex polytope in several
	key ways. Many discrete geometric objects can be viewed as an abstract polytope by considering their face-lattices,
	but there are also many new kinds of structures that have no immediate geometric analogue. 

	A \emph{flag} of an abstract polytope is a chain in the poset that contains one element of each rank.
	In many ways, it is more natural to work with the flags of a polytope rather than the faces themselves.
	For example, every automorphism (order-preserving bijection) of a polytope is completely determined by
	its effect on any single flag.
	
	\emph{Regular} polytopes are those for which the automorphism group acts transitively on the set of flags.
	The automorphism group of a regular polytope is a quotient of some string Coxeter group $[p_1, \ldots, p_{n-1}]$,
	and conversely, every sufficiently nice quotient of a string Coxeter group appears as the automorphism group
	of a regular polytope. Indeed, it is actually possible to reconstruct a regular polytope from its
	automorphism group, so that much of the study of regular polytopes is purely group-theoretic.
	
	We say that a regular polytope $\calP$ is of \emph{type} (or has \emph{Schl\"afli symbol}) 
	$\{p_1, \ldots, p_{n-1}\}$ if $[p_1, \ldots, p_{n-1}]$ is the minimal string Coxeter group that covers 
	the automorphism group of $\calP$, 
	in a way that $p_1, \ldots, p_{n-1}$ are the orders of the relevant generators.  
	There is an equivalent formulation of this property that is entirely combinatorial, and hence it is possible
	to define a Schl\"afli symbol for many non-regular polytopes, including chiral polytopes (see \cite{chiral}) and
	other two-orbit polytopes (see \cite{two-orbit}). Any polytope with a well-defined Schl\"afli symbol
	is said to be \emph{equivelar}.
	
	In \cite{smallest-regular}, the first author determined the smallest regular polytope (by number of flags)
	in each rank. To begin with, he showed that every regular polytope of type $\{p_1, \ldots, p_{n-1}\}$
	has at least $2p_1 \cdots p_{n-1}$ flags. Polytopes that meet this lower bound are called \emph{tight}.
	He then exhibited a family of tight polytopes, one in each rank, of type $\{4, \ldots, 4\}$. 
	Using properties of the automorphism groups of regular polytopes, he showed that these were the smallest
	regular polytopes in rank $n \geq 9$, and that in every other rank, the minimum was also attained by
	a tight polytope (with type or dual type $\{3\}$, $\{3,4\}$, $\{4,3,4\}$, $\{3,6,3,4\}$, $\{4,3,6,3,4\}$, 
	$\{3,6,3,6,3,4\}$ or $\{4,3,6,3,6,3,4\}$, respectively).
	
	The second author showed in \cite{tight-polytopes} that the bound on the number of flags extended to
	any equivelar polytope, regardless of regularity. Accordingly, it makes sense to extend the definition of
	tight polytopes to include any polytope of type $\{p_1, \ldots, p_{n-1}\}$ with $2p_1 \cdots p_{n-1}$ flags.
	An alternate formulation was proved as well, showing that an equivelar polytope is tight if and only if
	every face is incident with all faces that are 2 ranks higher.
	
	Tightness is a restrictive property, and not every Schl\"afli symbol is possible for a tight polytope.
	In order for there to be a tight polytope of type $\{p_1, \ldots, p_{n-1}\}$, it is necessary that
	no two adjacent values $p_i$ and $p_{i+1}$ are odd. Theorem 5.1 in \cite{tight-polytopes} shows that
	this condition is sufficient in rank 3. In higher ranks, the question of sufficiency is still open.

	Constructing non-regular polytopes in high ranks is difficult. In order to determine which
	Schl\"afli symbols are possible for a tight polytope, it is helpful to begin by considering only regular polytopes.
	If every $p_i$ is even, then Theorem 5.3 in \cite{smallest-regular} and Theorem 6.3 in \cite{tight-polytopes}
	show that there is a tight regular polytope of type $\{p_1, \ldots, p_{n-1}\}$. 
	Also the computational data from \cite{conder-atlas, atlas} led the second author to 
	conjecture that if $p$ is odd and $q > 2p$, there is no tight regular polyhedron of type $\{p, q\}$. 
	Although we are currently unable to prove this conjecture,
	we can show that for tight \emph{orientably}-regular polyhedra, if $p$ is odd then 
	$q$ must divide $2p$. Moreover, we are able to prove the following generalisation in higher ranks:
	
	\begin{theorem}
	\label{thm:tight-existence2}
	There is a tight orientably-regular polytope of type $\{p_1, \ldots, p_{n-1}\}$ if and only if 
	the integers $p_{i-1}$ and $p_{i+1}$ are even divisors of $2p_i$ whenever $p_i$ is odd. 
	\end{theorem}

\section{Background}
	
	Our background information is mostly taken from \cite[Chs. 2, 3, 4]{arp}, with a few small additions.

	\subsection{Definition of a polytope}

		Let $\calP$ be a ranked partially-ordered set, the elements of which are called \emph{faces}, 
		and suppose that the faces of $\calP$ range in rank from $-1$ to $n$. 
		We call each face of rank $j$ a \emph{$j$-face}, 
		and we say that two faces are \emph{incident} if they are comparable.
		We also call the $0$-faces, $1$-faces and $(n-1)$-faces the 
		\emph{vertices}, \emph{edges} and \emph{facets} of $\calP$, respectively. 
		A \emph{flag} is a maximal chain in $\calP$. We say that two flags are \emph{adjacent} if they differ 
		in exactly one face, and that they are \emph{$j$-adjacent} if they differ only in their $j$-faces. 

		If $F$ and $G$ are faces of $\calP$
		such that $F \leq G$, then the \emph{section} $G / F$ consists of those faces $H$ such that
		$F \leq H \leq G$. If $F$ is a $j$-face and $G$ is a $k$-face, then we say that the \emph{rank}
		of the section $G / F$ is $k \shm j \shm 1$.
		If removing $G$ and $F$ from the Hasse diagram of $G / F$ leaves us with
		a connected graph, then we say that $G/F$ is \emph{connected}. That is, for
		any two faces $H$ and $H'$ in $G/F$ (other than $F$ and $G$ themselves), there is a sequence of faces
		\[ H = H_0, H_1, \ldots, H_k = H' \]
		such that $F < H_i < G$ for $0 \leq i \leq k$ and the faces $H_{i-1}$ and $H_i$ are 
		incident for $1 \leq i \leq k$.
		By convention, we also define all sections of rank at most $1$ to be connected.
		
		We say that $\calP$ is an (\emph{abstract}) \emph{polytope of rank $n$}, or equivalently, 
		an \emph{$n$-polytope}, if it satisfies the following four properties: \\[-18pt] 
		\begin{enumerate}
		\item There is a unique greatest face $F_n$ of rank $n$, 
		and a unique least face $F_{-1}$ of rank $-1$. \\[-18pt]
		\item Each flag has $n \shp 2$ faces. \\[-18pt]
		\item Every section is connected. \\[-18pt]
		\item Every section of rank $1$ is a diamond --- that is, 
		whenever $F$ is a $(j \shm 1)$-face and $G$ is a $(j \shp 1)$-face for some $j$, 
		with $F < G$, there are exactly two $j$-faces $H$ with $F < H < G$. 
		\end{enumerate}
		
		Condition (d) is known as the {\em diamond condition.}  
		Note that this condition ensures that for  $0 \le j \le n \shm 1$, every flag $\Phi$ has a unique $j$-adjacent 
		flag, which we denote by $\Phi^j$.
		
		In ranks $-1$, $0$, and $1$, there is a unique polytope up to isomorphism. Abstract polytopes of rank $2$
		are also called \emph{abstract polygons}, and for each $2 \leq p \leq \infty$, there is
		a unique abstract polygon with $p$ vertices and $p$ edges, denoted by $\{p\}$.
		
		If $F$ is a $j$-face and $G$ is a $k$-face of a polytope with $F \leq G$, then the section $G/F$ itself is a
		($k \shm j \shm 1$)-polytope. We may identify a face $F$ with the section $F/F_{-1}$, and call the section 
		$F_n/F$ the \emph{co-face at $F$}. The co-face
		at a vertex $F_0$ is also called a \emph{vertex-figure} at $F_0$. 

		If $\calP$ is an $n$-polytope, $F$ is an $(i \shm 2)$-face of $\calP$, and $G$ is an $(i \shp 1)$-face of
		$\calP$ with $F < G$, then the section $G/F$ is an abstract polygon. 
		If it happens that for $1 \leq i \leq n-1$, each such section is the same polygon $\{p_i\}$, 
		no matter which $(i \shm 2)$-face $F$ and $(i \shp 1)$-face $G$ we choose, 
		then we say that $\calP$ has \emph{Schl\"{a}fli symbol $\{p_1, \ldots, p_{n-1}\}$}, 
		or that $\calP$ is \emph{of type $\{p_1, \ldots, p_{n-1}\}$}.
		%Moreover, if $\calP$ has a Schl\"{a}fli symbol, then 
		Also when this happens, we say that $\calP$ is \emph{equivelar}.
		
		All sections of an equivelar polytope are themselves equivelar polytopes. In particular,
		if $\calP$ is an equivelar polytope of type $\{p_1, \ldots, p_{n-1}\}$, then all its facets
		are equivelar polytopes of type $\{p_1, \ldots, p_{n-2}\}$, and all its vertex-figures
		are equivelar polytopes of type $\{p_2, \ldots, p_{n-1}\}$.
		
		Next, let $\calP$ and $\calQ$ be two polytopes of the same rank.
		A surjective function $\gamma: \calP \to \calQ$ is called a
		\emph{covering} if it preserves incidence of faces, ranks of faces, and adjacency of flags.
		If there exists a such covering $\gamma: \calP \to \calQ$, then we say that $\calP$ \emph{covers} $\calQ$.
		
		The \emph{dual} of a polytope $\calP$ is the polytope obtained by reversing the
		partial order. If $\calP$ is an equivelar polytope of type $\{p_1, \ldots, p_{n-1}\}$, then
		the dual of $\calP$ is an equivelar polytope of type $\{p_{n-1}, \ldots, p_1\}$.
		
	\subsection{Regularity}
	
		For polytopes $\calP$ and $\calQ$, an \emph{isomorphism} from $\calP$ to $\calQ$ is 
		an incidence- and rank-preserving bijection. % on the set of faces. 
		By connectedness and the diamond condition, 
		every polytope isomorphism is uniquely determined by its effect on a given flag. 
		An isomorphism from $\calP$ to itself is an \emph{automorphism} of $\calP$, and the group of 
		all automorphisms of $\calP$ is denoted by $\G(\calP)$. 
		We will denote the identity automorphism by $\varepsilon$. 
		
		We say that $\calP$ is \emph{regular} if the natural action of $\G(\calP)$ on the flags 
		of $\calP$ is transitive (and hence regular, in the sense of being sharply-transitive).
		For convex polytopes, this definition is equivalent to any of the usual definitions of regularity.
		
		Now let $\calP$ be any regular polytope, and choose a flag $\Phi$, which we call a \emph{base flag}. 
		Then the automorphism
		group $\G(\calP)$ is generated by the \emph{abstract reflections} $\rho_0, \ldots, \rho_{n-1}$,
		where $\rho_i$ maps $\Phi$ to the unique flag $\Phi^i$ that is $i$-adjacent to $\Phi$. 
		These generators satisfy $\rho_i^2 = \varepsilon$ for all $i$, and $(\rho_i \rho_j)^2 = \varepsilon$ for all $i$ 
		and $j$ such that $|i - j| \geq 2$.
		Every regular polytope is equivelar, and if its Schl\"{a}fli symbol is $\{p_1, \ldots, p_{n-1}\}$,
		then the order of each $\rho_{i-1} \rho_i$ is $p_i$.
		Note that if $\calP$ is a regular polytope of type $\{p_1, \ldots, p_{n-1}\}$, then $\G(\calP)$ is
		a quotient of the string Coxeter group $[p_1, \ldots, p_{n-1}]$, which is the abstract group
		generated by $n$ elements $x_0,\dots,x_{n-1}$ subject to the defining relations 
		$x_i^2 = 1$, $(x_{i-1} x_i)^{p_i} = 1$, and $(x_i x_j)^2 = 1$
		whenever $|i - j| \geq 2$.
		
		Next, if $\G$ is any group generated by elements $\rho_0, \ldots, \rho_{n-1}$, 
		we define $\G_I = \langle \rho_i \mid i \in I \rangle$ for each subset $I$ of 
		the index set $\{0, 1, \ldots, n-1\}$. If $\G$ is the automorphism group $\G(\calP)$ 
		of a regular polytope $\calP$, then these subgroups satisfy the following condition, 
		known as the {\em intersection condition\/}:
		\begin{equation}
		\label{eq:reg-int}
		\G_I \cap \G_J = \G_{I \cap J} 
		\;\; \textrm{ for all } I,J \subseteq \{0,1, \ldots, n-1\}.
		\end{equation}

		More generally, if $\G$ is any group generated by elements $\rho_0, \ldots, \rho_{n-1}$ 
		of order $2$ such that $(\rho_i \rho_j)^2 = 1$ whenever $|i - j| \geq 2$, then
		we say that $\G$ is a \emph{string group generated by involutions}, and abbreviate this 
		to say that $\G$ is an \emph{sggi}. 
		If the sggi $\G$ also satisfies the intersection condition (\ref{eq:reg-int}) given above, 
		then we call $\G$ a \emph{string C-group}. 
		
		There is a natural way of building a regular polytope $\calP(\G)$ from a string
		C-group $\G$ such that $\G(\calP(\G)) \cong \G$ and $\calP(\G(\calP)) \cong \calP$. 
		In particular, the $i$-faces of $\calP(\G)$ are taken to be the cosets of the subgroup 
		$\G_i = \langle \rho_j \mid j \neq i \rangle$, 
		with incidence of faces $\G_i \varphi $ and $\G_j \psi$ given by 
		\[\G_i \varphi \leq \G_j \psi \ \ \hbox{ if and only if } \  \ i \leq j \ \hbox{ and } \ \G_i \varphi \cap
		\G_j \psi \neq \emptyset.\] 
		This construction is also easily applied when $\G$ is any sggi (not necessarily a string C-group),
		but in that case, the resulting poset is not always a polytope.

		The following theory from \cite{arp} helps us determine when a given sggi is a string C-group:
		
		\begin{proposition}
		\label{prop:quo-crit}
		Let $\G = \langle \rho_0, \ldots, \rho_{n-1} \rangle$ be an sggi, and $\Lambda = \langle \lambda_0, \ldots, 
		\lambda_{n-1} \rangle$ a string C-group. If there is a homomorphism
		$\pi: \G \to \Lambda$ sending each $\s_i$ to $\lambda_i$, and if $\pi$ is one-to-one on the subgroup
		$\langle \rho_0, \ldots, \rho_{n-2} \rangle$ or the subgroup $\langle \rho_1, \ldots, \rho_{n-1} \rangle$, 
		then $\G$ is a string C-group.
		\end{proposition}
		
		\begin{proposition}
		\label{prop:facet-vfig}
		Let $\G = \langle \rho_0, \ldots, \rho_{n-1} \rangle$ be an sggi. If $\langle \rho_0, \ldots, \rho_{n-2} \rangle$
		and $\langle \rho_1, \ldots, \rho_{n-1} \rangle$ are string C-groups, and $\langle \rho_0, \ldots, \rho_{n-2} \rangle
		\cap \langle \rho_1, \ldots, \rho_{n-1} \rangle \subseteq \langle \rho_1, \ldots, \rho_{n-2} \rangle$, then
		$\G$ is a string C-group.
		\end{proposition}

		Given a regular $n$-polytope $\calP$ with automorphism 
		group $\G = \langle \rho_0, \ldots, \rho_{n-1} \rangle$, we define 
		the \emph{abstract rotations} $\s_1,\dots,\s_{n-1}$ 
		by setting $\s_i = \rho_{i-1} \rho_i$  for $1 \leq i \leq n-1$. 
		Then the subgroup $\langle \s_1, \ldots, \s_{n-1} \rangle$
		of $\G(\calP)$ is denoted by $\G^+(\calP)$, and called the \emph{rotation subgroup of $\calP$}.
		The index of $\G^+(\calP)$ in $\G(\calP)$ is at most $2$, and when the index is exactly $2$,
		then we say that $\calP$ is \emph{orientably-regular}. Otherwise, if $\G^+(\calP) = \G(\calP)$,
		then we say that $\calP$ is \emph{non-orientably-regular}. 
		(This notation comes from the study of regular maps.) %; see~\cite{rcm} for example.) 
		A regular polytope $\calP$ is orientably-regular if and only if $\G(\calP)$ has a presentation 
		in terms of the generators $\rho_0, \ldots, \rho_{n-1}$ such that all of the relators have even length. Note
		that every section of an orientably-regular polytope is itself orientably-regular.
		
	\subsection{Flat and tight polytopes}
	
		The theory of abstract polytopes accomodates certain degeneracies not present in the study of
		convex polytopes. For example, the face-poset of a convex polytope is a lattice 
		(which means that any two elements
		have a unique supremum and infimum), but this need not be the case with abstract polytopes.
		The simplest abstract polytope that is not a lattice is the digon $\{2\}$, in which 
		both edges are incident with both vertices. This type of degeneracy
		can be generalised as follows. If $\calP$ is an $n$-polytope, and 
		$0 \leq k < m \leq n-1$, then we say that $\calP$ is
		\emph{$(k, m)$-flat} if every one of its $k$-faces is incident with every one of its $m$-faces.
		If $\calP$ has rank $n$ and is $(0, n \shm 1)$-flat, then we also say simply that $\calP$ is a 
		\emph{flat polytope}. Note that if $\calP$ is $(k, m)$-flat,
		then $\calP$ must also be $(i, j)$-flat whenever $0 \leq i \leq k < m \leq j \leq n \shm 1$. 
		In particular, if $\calP$ is $(k, m)$-flat, then it is also flat. 
		%for any $0 \leq k < m \leq n-1$, then it is also flat (that is, $(0, n-1)$-flat).

		We will also need the following, taken from \cite[Lemma 4E3]{arp}: 
		
		\begin{proposition}
		\label{prop:4e3}
		Let $\calP$ be an $n$-polytope, and let $0 \leq k < m < i \leq n \shm 1$. If each $i$-face of $\calP$ is 
		$(k,m)$-flat, then $\calP$ is also $(k, m)$-flat. Similarly, if $0 \leq i < k < m \leq n \shm 1$ and each 
		co-$i$-face of $\calP$ is $(k \shm i \shm 1, m \shm i \shm 1)$-flat, then $\calP$ is $(k, m)$-flat.
		\end{proposition}

		It is easy to see that the converse is also true. 
		In other words, if $\calP$ is $(k, m)$-flat, then for $i > m$ each $i$-face of
		$\calP$ is $(k,m)$-flat, and for $j < k$ each co-$j$-face of $\calP$ is 
		$(k \shm j \shm 1, m \shm j \shm 1)$-flat.
		
		\smallskip
		Next, we consider tightness. 
		Every equivelar polytope $\calP$ of type $\{p_1, \ldots, p_{n-1}\}$
		has at least $2p_1 \cdots p_{n-1}$ flags, by \cite[Proposition 3.3]{tight-polytopes}. 
		Whenever $\calP$ has exactly this number of flags, we say that $\calP$ is \emph{tight}. 
		It is clear that $\calP$ is tight if and only if its dual is tight, 
		and that in a tight polytope, every section of rank $3$ or more is tight. 
		Also we will need the following, taken from \cite[Theorem 4.4]{tight-polytopes}:
		
		\begin{theorem}
		\label{thm:flat-is-tight}
		Let $n \geq 3$ and let $\calP$ be an equivelar $n$-polytope. Then $\calP$ is tight if and only if it is
		$(i,i+2)$-flat for $0 \leq i \leq n-3$.
		\end{theorem}

		Later in this paper we will build polytopes inductively, and for that, the following approach is useful.
		We say that the regular $n$-polytope $\calP$ has the \emph{flat amalgamation property} (or FAP) 
		with respect to its $k$-faces, if adding the relations $\rho_i = \eps$ for $i \geq k$ to $\G(\calP)$ 
		yields a presentation for $\langle \rho_0, \ldots, \rho_{k-1} \rangle$. 
		Similarly, we say that $\calP$ has the FAP with respect
		to its co-$k$-faces if adding the relations $\rho_i = \eps$ for $i \leq k$ yields a presentation
		for $\langle \rho_{k+1}, \ldots, \rho_{n-1} \rangle$. 
		
		We will also use the following, taken from \cite[Theorem 4F9]{arp}:
		
		\begin{theorem}
		\label{thm:fap}
		Suppose $m, n \geq 2$, and $0 \leq k \leq m-2$ where $k \geq m-n$. 
		Let $\calP_1$ be a regular $m$-polytope,
		and let $\calP_2$ be a regular $n$-polytope such that the co-k-faces of $\calP_1$ are isomorphic to the
		$(m \shm k \shm 1)$-faces of $\calP_2$. 
		Also suppose that $\calP_1$ has the FAP with respect to its co-$k$-faces, and that
		$\calP_2$ has the FAP with respect to its $(m \shm k \shm 1)$-faces. 
		Then there exists a regular $(k \shp n \shp 1)$-polytope
		$\calP$ such that $\calP$ is $(k,m)$-flat, and the $m$-faces of $\calP$ are isomorphic to $\calP_1$, while
		the co-$k$-faces of $\calP$ are isomorphic to $\calP_2$. Furthermore, $\calP$ has the FAP with respect
		to its $m$-faces and its co-$k$-faces.
		\end{theorem}
		
\section{Tight orientably-regular polyhedra}

	We now consider the values of $p$ and $q$ for which there is a tight orientably-regular polyhedron
	of type $\{p, q\}$. 
	By \cite[Proposition 3.5]{tight-polytopes}, there are no tight polyhedra of type
	$\{p, q\}$ when $p$ and $q$ are both odd. 
	Also by \cite[Theorem 5.3]{smallest-regular} and \cite[Theorem 6.3]{tight-polytopes}, 
	if $p$ and $q$ are both even then there exists a tight orientably-regular polyhedron whose
	automorphism group is the quotient of the Coxeter group $[p, q]$ obtained by adding 
	the extra relation $(x_0 x_1 x_2 x_1)^2 = 1$.
	We note, however, that this might not the only tight orientably-regular polyhedron of the 
	given type; for example, there are two of type $\{4, 8\}$ that are non-isomorphic (see \cite{conder-atlas}).

	When $p$ is odd and $q$ is even (or vice-versa), the situation is more complicated. 
	Evidence from \cite{conder-atlas} and \cite{atlas} led the second author to conjecture
	that there are no tight regular polyhedra of type $\{p, q\}$ if $p$ is odd and $q > 2p$ (see \cite{tight-polytopes}).
	We will show that this is true in the orientably-regular case. In fact, we will
	prove something stronger, namely that $q$ must divide $2p$.
	
	We start by showing that if $p$ is odd and $q$ is an even divisor of $2p$, then there is a tight
	orientably-regular polyhedron of type $\{p, q\}$. 
	To do this, we define $\G(p, q)$ to be the group 
	$\gp{\rho_0, \rho_1, \rho_2}{\rho_0^{\ 2},\rho_1^{\ 2}, \rho_2^{\ 2}, 
	(\rho_0 \rho_1)^p, (\rho_1 \rho_2)^q, (\rho_0 \rho_2)^2,
	(\rho_0 \rho_1 \rho_2 \rho_1 \rho_2)^2}$,  
	which is obtainable by adding one extra relator to the Coxeter group $[p, q]$. 
	
	\begin{theorem}
	\label{thm:q-div-2p}
	Let $p \geq 3$ be odd, and let $q$ be an even divisor of $2p$. Then there is a tight orientably-regular
	polyhedron $\calP$ of type $\{p, q\}$ such that $\G(\calP) \cong \G(p, q)$.
	\end{theorem}
	
	\begin{proof}
	Let $\G(p, q) = \langle \rho_0, \rho_1, \rho_2 \rangle$. In light of the construction in
	Section 2.2, all we need to do is show that $\G(p, q)$ is a string C-group of order $2pq$, 
	in which the order of $\rho_0 \rho_1$ is $p$ and the order of $\rho_1 \rho_2$ is $q$.
	
	First, note that the element $\omega = (\rho_1 \rho_2)^2$ generates a cyclic normal subgroup $N$ 
	of $\G(p, q)$, since each of $\rho_1$ and $\rho_2$ conjugates $\omega$ to its inverse, 
	and the extra relation $(\rho_0 \rho_1 \rho_2 \rho_1 \rho_2)^2 = 1$ implies that $\rho_0$ 
	does the same. 
	Factoring out $N$ gives quotient $\G(p, 2)$, in which the extra 
	relation $(\rho_0 \rho_1 \rho_2 \rho_1 \rho_2)^2 = 1$ is redundant. 
	In fact $\G(p, 2)$ is isomorphic to the string Coxeter group $[p, 2]$, 
	which is an extension of the dihedral group of order $2p$, and has order $4p$. 
         
         In particular, $\G(p, q)$ covers $\G(p, 2) \cong [p,2]$, and it follows that 
         $\rho_0 \rho_1$ has order $p$ (rather than some proper divisor of $p$). 
	Also the cover from $\G(p, q)$ to $\G(p, 2)$ is one-to-one on $\langle
	\rho_0, \rho_1 \rangle$, and so by \pref{quo-crit}, we find that $\G(p, q)$ is a string C-group.

	Next, we observe that the dihedral 
	group $D_q = \gp{y_1, y_2}{y_1^{\ 2}, y_2^{\ 2}, (y_1 y_2)^q}$ 
	is a quotient of $\G(p, q)$, via an epimorphism taking $\rho_1 \mapsto y_1$, 
	$\,\rho_2 \mapsto y_2\,$ and $\,\rho_0 \mapsto (y_1 y_2)^{p-1} y_1$.  
	(Note that the defining relations for $\G(p, q)$ are satisfied by their images 
	in $D_q$, since $(y_1 y_2)^{p-1} y_1$ has order $2$, 
	the order of $(y_1 y_2)^{p-1}$ divides $p$ (as $p \shm 1$ is even and $q$ divides $2p$), 
	the order of $(y_1 y_2)^{p-1} y_1 y_2 = (y_1 y_2)^{p}$ divides $2$ (as $q$ divides $2p$), 
	and $(y_1 y_2)^{p-1} y_2 y_1 y_2$ has order $2$.) 
	In particular, the image of $\rho_1 \rho_2$ is $y_1 y_2$, which has order $q$, 
	and hence $\rho_1 \rho_2$ has order $q$. 
	
	Finally, $|\G(p, q)| = |\G(p, q)/N||N| = |\G(p, 2)||N| = 4p(q/2) = 2pq$, 
	since $N = \langle (\rho_1 \rho_2)^2 \rangle$ has order $q/2$. 
	\end{proof}

	We will show that in fact, the only tight orientably-regular polyhedra of type $\{p, q\}$ with $p$
	odd are those given in \tref{q-div-2p}. 
	We proceed with the help of a simple lemma.
	
	\begin{lemma}
	\label{lem:s2-normal}
	Let $\calP$ be an orientably-regular polyhedron of type $\{p, q\}$, with $p$ odd, and with 
	automorphism group $\G(\calP)$ generated by the reflections $\rho_0, \rho_1, \rho_2$. 
	If $\omega = (\rho_1 \rho_2)^2 = \s_2^{\ 2}$ generates a normal subgroup of $\G^+(\calP)$, 
	then $\omega$ is central, and $q$ divides $2p$.
	\end{lemma}
	
	\begin{proof}
	For simplicity, let $x = \s_1 = \rho_0 \rho_1$ and $y = \s_2 = \rho_1 \rho_2$, 
	so that $xy =  \rho_0 \rho_2$ and hence $x^p = y^q = (xy)^2 = 1$, and also $\omega = y^2$.
	By hypothesis, $\langle y^2 \rangle$ is normal, and so $x y^2 x^{-1} = y^{2k}$ for some $k$. 
	%while of course $y y^2 y^{-1} = y^2$. 
	It follows that $y^2 = x^p y^2 x^{-p} = y^{2k^p}$ and 
	that $y^2 = (xy)^2 y^2 (xy)^{-2} = y^{2k^2}$, 
	and therefore $2 \equiv 2k^2 \equiv 2k^p$ mod $q$. 
	Then also $2 \equiv 2k^2k^{p-2} \equiv 2k^{p-2}$ mod $q$, 
	and by induction $2 \equiv 2k^p \equiv 2k^{p-2} \equiv \cdots \equiv 2k$ mod $q$, since $p$ is odd. 
	Thus $x y^2 x^{-1} = y^{2k} = y^2$, 
	and so $\omega = y^2$ is central.
	Moreover, since $y^2$ commutes with $x$ (which has order $p$) and $xy = y^{-1}x^{-1}$, 
	we find that 
	$y^{2p} = x^p y^{2p} = (xy^2)^p = (y^{-1}x^{-1}y)^p = y^{-1}x^{-p}y = 1$, 
	and so $2p$ is a multiple of $q$. 
	\end{proof}

	We also utilise a connection between tight polyhedra and regular Cayley maps, 
	as is explained in \cite{rcm}. 
	Specifically, suppose that the finite group $G$ is generated by two non-involutory 
	elements $x$ and $y$ such that  $xy$ has order $2$, and that $G$ can be written
	as $AY$ where $Y = \langle y \rangle$ is core-free in $G$, and $A$ is a subgroup of $G$ 
	such that $A \cap Y = \{1\}$.  Then $G$ is the group $\G^+(M)$ of orientation-preserving 
	automorphisms group of a regular Cayley map $M$ for the group $A$.
	Furthermore, this map $M$ is reflexible if and only if $G$ admits an automorphism 
	taking $x \mapsto xy^2$ ($= y^{-1}x^{-1}y$) and $y \mapsto y^{-1}$.

	\begin{theorem}
	\label{thm:orient-reg-tight}
	Let $p \geq 3$ be odd. If $\calP$ is a tight orientably-regular polyhedron of type $\{p, q\}$, 
	then $q$ is an even divisor of $2p$, and $\G(\calP)$ is isomorphic to $\G(p, q)$.
	\end{theorem}

	\begin{proof}
	Let $G = \G^+(\calP)$, and let $\s_1 = \rho_0 \rho_1$ and $\s_2 = \rho_1 \rho_2$ be its 
	standard generators. 
	Also take $F = \langle \s_1 \rangle$ and $V = \langle \s_2 \rangle$, which are the stabilisers 
	in $\G^+(\calP)$ of a 2-face and incident vertex of $\calP$. 
	Then $F \cap V = \langle \eps \rangle$ since $\calP$ is a polytope,
	and $G = FV$ since $\calP$ is tight. 
	
	Now, let $N$ be the core of $V$ in $G$ 
	(which is the largest normal subgroup of $G$ contained in $V$), 
	and let $\overline{G} = G/N$,  $\overline{V} = V/N$ and $\overline{F} = FN/N$. 
	Then $\overline{G}  = \overline{V} \, \overline{F}$, and $\overline{V} \cap \overline{F}$ is trivial,  
	and also $\overline{V}$ is core-free. 
	Thus $\overline{G}$ is the orientation-preserving automorphism group of 
	a regular Cayley map $M$ for the cyclic group $\overline{F}$.  
	Furthermore, since $\calP$ is an orientably-regular polyhedron, 
	the group $G = \G^+(\calP)$ has an automorphism taking $\s_1 \mapsto \s_1 \s_2^2$ 
	and $\s_2 \mapsto\s_2^{-1}$, and $\overline{G}$ has the analogous property. 
	Hence $M$ is reflexible. 
	
	On the other hand, by \cite[Theorem 3.7]{rcm} we know that the only reflexible 
	regular Cayley map for a cyclic group of odd order $p$ is the equatorial map on the 
	sphere, with $p$ vertices of valence $2$. 
	Thus $|\overline{V}| = 2$, and so $q = |V|$ is even, and $N = \langle \s_2^2 \rangle$.
	
	In particular, $\langle \s_2^2 \rangle$ is a normal subgroup of $\G^+(\calP)$, 
	and hence by \lref{s2-normal} we also find that $q$ divides $2p$, and that 
	$\s_2^2$ is central. 
	But $\s_2^2 = (\rho_1 \rho_2)^2$ is inverted under conjugation by $\rho_1$, 
	and now centralised by $\s_1 = \rho_0 \rho_1$, and therefore also inverted 
	under conjugation by $\rho_0$.
	Hence the relation $(\rho_0 \rho_1 \rho_2 \rho_1 \rho_2)^2 = \eps$ holds in $\G(\calP)$, 
	so $\G(\calP)$ is a quotient of $\G(p,q)$.
	Then finally, since $\calP$ is tight we have $|\G(\calP)| = 2pq = |\G(p,q)|$, 
	and it follows that $\G(\calP) \cong \G(p, q)$.
	\end{proof}

	Combining \tref{q-div-2p} with \tref{orient-reg-tight} and \cite[Theorem 6.3]{tight-polytopes}, 
	we can now draw the following conclusion:
	
	\begin{theorem}
	\label{thm:tight-polyhedra}
	There is a tight orientably-regular polyhedron of type $\{p, q\}$ if and only if 
	one of the following is true$\hskip 1pt :$ \\[-20 pt] 
	\begin{enumerate}
	\item $p$ and $q$ are both even, or \\[-22 pt] 
	\item $p$ is odd and $q$ is an even divisor of $2p$, or \\[-22 pt] 
	\item $q$ is odd and $p$ is an even divisor of $2q$.
	\end{enumerate}
	\end{theorem}
	
\section{Tight orientably-regular polytopes in higher ranks}

	Now that we know which Schl\"afli symbols appear among tight orientably-regular polyhedra, 
	we can proceed to classify the tight orientably-regular polytopes of arbitrary rank. 
	We will say that the $(n \shm 1)$-tuple $(p_1, \ldots, p_{n-1})$ is \emph{admissible} if $p_{i-1}$ and
	$p_{i+1}$ are even divisors of $2p_i$ whenever $p_i$ is odd.
	
	\begin{theorem}
	\label{thm:tight-nec}
	If $\calP$ is a tight orientably-regular polytope of type $\{p_1, \ldots, p_{n-1}\}$, 
	then the $(n \shm 1)$-tuple $(p_1, \ldots, p_{n-1})$ is admissible.
	\end{theorem}
	
	\begin{proof}
	If $\calP$ is tight and orientably-regular, then by \cite[Proposition 3.8]{tight-polytopes} 
	we know that all of its sections of rank $3$ are tight and orientably-regular.
	Hence in particular, if $p_i$ is odd then the sections of $\calP$ of type $\{p_{i-1}, p_i\}$ and $\{p_i, p_{i+1}\}$
	are tight and orientably-regular. The rest now follows from \tref{orient-reg-tight}.
	\end{proof}

	We will prove that this necessary condition is also sufficient, which will then complete the proof of
	\tref{tight-existence2}. 
	We do this by constructing the automorphism group of a tight orientably-regular polytope of the given type.
	
	Let $(p_1, \ldots, p_{n-1})$ be an admissible $(n \shm 1)$-tuple. 
	Then we define the group $\G(p_1, \ldots, p_{n-1})$ to be the quotient of the 
	string Coxeter group $[p_1, \ldots, p_{n-1}]$ obtained by adding $n \shm 2$ 
	extra relations $r_1 = \cdots = r_{n-2} = 1$, where 
		\[ r_i = \begin{cases}
	(x_{i-1} x_i x_{i+1} x_i)^2 & \textrm{if $p_i$ and $p_{i+1}$ are both even, or } \\
	(x_{i-1} x_i x_{i+1} x_i x_{i+1})^2 & \textrm{if $p_i$ is odd and $p_{i+1}$ is even, or } \\
	(x_{i+1} x_i x_{i-1} x_i x_{i-1})^2 & \textrm{if $p_i$ is even and $p_{i+1}$ is odd.}
	\end{cases}
	\]
	Note that if $n = 3$, this definition of $\G(p_1, p_2)$ 
	coincides with the one  in the previous section. 
	
	Also let $\calP(p_1, \ldots, p_{n-1})$ be the poset obtained from
	$\G(p_1, \ldots, p_{n-1})$, using the construction in Section 2.2. 
	We will show that $\G(p_1, \ldots, p_{n-1})$ is a string C-group of order $2p_1p_2\dots p_{n-1}$, 
	and then since every relator of $\G(p_1, \ldots, p_{n-1})$ has even length, 
	it follows that $\calP(p_1, \ldots, p_{n-1})$ is a tight orientably-regular polytope 
	of type $\{p_1, \ldots, p_{n-1}\}$. 

	We start by considering the order of $\G(p_1, \ldots, p_{n-1})$.

	\begin{proposition}
	\label{prop:norm-subgp}
	Let $p_i$ be even. Then every element of $\G(p_1, \ldots, p_{n-1})$ either commutes with
	$(x_{i-1} x_i)^2$ or inverts it by conjugation.  
	In particular, the square of every element of $\G(p_1, \ldots, p_{n-1})$ commutes with
	$(x_{i-1} x_i)^2$. 
	\end{proposition}
	
	\begin{proof}
	Let $\omega = (x_{i-1} x_i)^2$. If $j \leq i-3$ or $j \geq i+2$, then $x_j$ commutes with both $x_{i-1}$ and $x_i$, 
	and so commutes with $\omega$. 
	Also it is clear that $x_{i-1}$ and $x_i$ both conjugate $\omega$ to $\omega^{-1}$, 
	so it remains to consider only $x_{i-2}$ and $x_{i+1}$. 
	Now since $p_i$ is even, the relator $r_{i-1}$ is either
	$(x_{i-2} x_{i-1} x_i x_{i-1})^2$ or $(x_{i-2} x_{i-1} x_i x_{i-1} x_i)^2$.
	In the first case, $x_{i-2}$ commutes with $x_{i-1} x_i x_{i-1}$ and 
	hence with $(x_{i-1} x_i x_{i-1}) x_i = \omega$, 
	while in the second case, we have $(x_{i-2} \omega)^2 = 1$ and
	so $x_{i-2}$ conjugates $\omega$ to $\omega^{-1}$.  
	Similarly, the relator $r_i$ is $(x_{i-1} x_i x_{i+1} x_i)^2$ or $(x_{i+1} x_i x_{i-1} x_i x_{i-1})^2$, 
	and in these two cases we find that $x_{i+1} \omega x_{i+1} = \omega$ or $\omega^{-1}$, 
	respectively. 
	Thus every generator $x_j$ of $\G(p_1, \ldots, p_{n-1})$ either commutes with
	$(x_{i-1} x_i)^2$ or inverts it by conjugation, and it follows that the same is true for every element 
	of $\G(p_1, \ldots, p_{n-1})$.  The rest follows easily. 
	\end{proof}

	\begin{proposition}
	\label{prop:cox-2}
	Let $(p_1, \ldots, p_{n-1})$ be an admissible $(n \shm 1)$-tuple  
	with the property that for every $i$ strictly between $1$ and $n \shm 1$, 
	either $p_i = 2$ or $p_{i-1} = p_{i+1} = 2$.
	Then %$\G(p_1, \ldots, p_{n-1}) \cong [p_1, \ldots, p_{n-1}]$. In particular,
	$y_i = x_{i-1}x_i$ has order $p_i$ for all $i$, 
	and $|\G(p_1, \ldots, p_{n-1})| = 2p_1 \cdots p_{n-1}$.
	\end{proposition}
	
	\begin{proof}
	%Let $\G(p_1, \ldots, p_{n-1}) = \langle x_0, \ldots, x_{n-1} \rangle$.
	We use induction on $n$, and the observation that if $p_j = 2$, 
	then $1 = (x_{j-1}x_j)^2$, so that $x_{j-1}$ commutes with $x_j$, 
	and therefore $\langle x_0, \ldots, x_{j-1} \rangle$ centralises $\langle x_j, \ldots, x_{n-1} \rangle$. 
	%As a result, the relations $r_{i-1} = \eps$ and $r_i = \eps$
	%are rendered redundant and may be removed. Since, for every $i$, either $p_i = 2$ or $p_{i-1} = p_{i+1} = 2$,
	%it follows that every relation of the form $r_i = \eps$ is redundant, and therefore $\G(p_1, \ldots, p_{n-1}) =
	%[p_1, \ldots, p_{n-1}]$. Furthermore, 
	%
	Now if $p_1 = 2$, then 
	$\G(p_1,p_2, \ldots, p_{n-1}) = \G(2,p_2, \ldots, p_{n-1}) 
	  \cong \langle x_0 \rangle \times \G(p_2, \ldots, p_{n-1})$, 
	and so 
	\[|\G(p_1, \ldots, p_{n-1})| = 2|\G(p_2, \ldots, p_{n-1})| 
	 = 4p_2 \cdots p_{n-1} = 2p_1p_2 \cdots p_{n-1}.\] 
	Otherwise $p_2 = 2$ and %then 
	$\G(p_1,p_2, \ldots, p_{n-1}) = \G(p_1,2,p_3,\ldots, p_{n-1}) 
	  \cong \G(p_1) \times \G(p_3, \ldots, p_{n-1})$, so
	 \[|\G(p_1, \ldots, p_{n-1})| = |\G(p_1)| \, |\G(p_3, \ldots, p_{n-1})| 
	 = 2p_1 \, 2p_3 \cdots p_{n-1} = 2p_1p_2 \cdots p_{n-1}.\] 
	The claim about the orders of the elements $y_i = x_{i-1}x_i$ 
	follows easily by induction as well. 
	\end{proof}
	
	${}$\\[-30pt]
	\begin{lemma}
	\label{lem:tight-gp}
	Let $q_i$ be the order of $x_{i-1} x_i$ in $\G(p_1, \ldots, p_{n-1})$, 
	for $1 \leq i \leq n \shm 1$. 
	Then $q_i = p_i$ whenever $p_i$ is odd, 
	and also $|\G(p_1, \ldots, p_{n-1})| = 2q_1 \cdots q_{n-1}$, which divides $2p_1 \cdots p_{n-1}.$ 
	\end{lemma}
	
	\begin{proof}
	Let $ k_i = p_i$ when $p_i$ is odd, or $2$ when $p_i$ is even. 
	Then since $k_i$ divides $p_i$ for all $i$, there exists an 
	epimorphism $\pi: \G(p_1, \ldots, p_{n-1}) \to  \G(k_1, \ldots, k_{n-1})$. 
	Also the $(n \shm 1)$-tuple $(k_1, \ldots, k_{n-1})$ is admissible, 
	and indeed $k_{i-1}$ and $k_{i+1}$ are both $2$ whenever $k_i$ is odd 
	(since $p_{i-1}$ and $p_{i+1}$ are both even whenever $p_i$ is odd). 
	Thus $(k_1, \ldots, k_{n-1})$ satisfies the hypotheses of Proposition~\ref{prop:cox-2}, 
	and so $|\G(k_1, \ldots, k_{n-1})| = 2k_1 \cdots k_{n-1}$.
	
	Moreover, Proposition~\ref{prop:cox-2} tells us that when $p_i$ is odd, the order 
	of the image of $x_i$ in $\G(k_1, \ldots, k_{n-1})$ is $k_i = p_i$, and so $q_i = p_i$;   
	on the other hand, if $p_i$ is even, then the order of the image of $x_i$ 
	in $\G(k_1, \ldots, k_{n-1})$ is $k_i = 2$, and so $q_i$ is even in that case. 
	
	Now the kernel of the epimorphism $\pi$ is the smallest normal subgroup 
	of $\G(p_1, \ldots, p_{n-1})$ containing the elements $(x_{i-1} x_i)^2$ 
	for those $i$ such that $p_i$ is even.  By Proposition~\ref{prop:norm-subgp}, 
	however, the subgroup $N$ generated by these elements is normal 
	in $\G(p_1, \ldots, p_{n-1})$, and abelian. 
	Hence in particular, $N = \ker\pi$, and also by the intersection condition, 
	$|N|$ is the product of the numbers $q_i/2$ over all $i$ for which $p_i$ is even. 
	Thus $|\G(p_1, \ldots, p_{n-1})| = 2q_1 \cdots q_{n-1}$.
	\end{proof}
	
	In order to use \tref{fap} to build our tight regular polytopes recursively, 
	we need two more observations.  
	The first concerns the flat amalgamation property (FAP): 
	
	\begin{proposition}
	\label{prop:fap}
	If $p_2$ is even, then $\calP(p_1, \ldots, p_{n-1})$ has the FAP with respect to its $2$-faces, and if
	$p_{n-2}$ is even, then $\calP(p_1, \ldots, p_{n-1})$ has the FAP with respect to its co-$(n \shm 3)$-faces.
	\end{proposition}
	
	\begin{proof}
	Let $p_2$ be even, and consider the effect of killing the generators 
	$x_i$ of $\G(p_1, \ldots, p_{n-1})$, for $i \geq 2$ 
	(that is, by adding the relations $x_i = 1$ to the presentation for $\G(p_1, \ldots, p_{n-1})$).  
	Each of the relators $r_3, \ldots, r_{n-2}$ contains only generators $x_i$ with $i \geq 2$, 
	so becomes redundant, and may be removed. 
	The relator $r_2$ reduces to $x_1^{\,2}$ or $x_1^{\,4}$, 
	while $r_1$ reduces to $(x_0 x_1^2)^2$, which is equivalent to $x_0^2$,
	and hence all of these become redundant too. 
	Thus adding the relations $x_i = 1$ to $\G(p_1, \ldots, p_{n-1})$
	has the same effect as adding the relations $x_i = 1$ to the string 
	Coxeter group $[p_1, \ldots, p_{n-1}]$. 
	It is easy to see that this gives the quotient group 
	with presentation $\gp{x_0, x_1}{x_0^2, x_1^2, (x_0 x_1)^{p_1}}$, which is
	the automorphism group of the $2$-faces of $\calP(p_1, \ldots, p_{n-1})$. 
	Thus $\calP(p_1, \ldots, p_{n-1})$ has the FAP with respect to its $2$-faces. 
	The second claim can be proved by a dual argument.
	\end{proof}
	
	\begin{proposition}
	\label{prop:tight-recursive}
	Let $\calP$ be an equivelar $n$-polytope with tight $m$-faces and tight co-$k$-faces, 
	where $m \geq k+3$.  Then $\calP$ is tight.
	\end{proposition}
	
	\begin{proof}
	Since the $m$-faces are tight, they are $(i, i \shp 2)$-flat for $0 \leq i \leq m \shm 3$, 
	by \tref{flat-is-tight}, and then by \pref{4e3}, the polytope $\calP$ is  $(i, i+2)$-flat 
	for $0 \leq i \leq m \shm 3$.
	Similarly, the co-$k$-faces are $(i, i \shp 2)$-flat for $0 \leq i \leq n \shm k \shm 4$, 
	and $\calP$ is $(i, i+2)$-flat for $k \shp 1 \leq i \leq n \shm 3$.
	Finally, since $m \geq k+3$, we see that $\calP$ is $(i, i+2)$-flat for $0 \leq i \leq n \shm 3$, 
	and again \tref{flat-is-tight} applies, to show that $\calP$ is tight.
	\end{proof}

	We can now prove the following.
	
	\begin{theorem}
	\label{thm:tight-existence}
	Let  $(p_1, \ldots, p_{n-1})$ be an admissible $(n \shm 1)$-tuple, with $n \geq 4$. 
	Also suppose that $p_{i-1}$ and $p_{i+1}$ are both even, for some $i$ $($with $2 \leq i \leq n \shm 2)$.
	If $\calP(p_1, \ldots, p_i)$ is a tight orientably-regular polytope of type $\{p_1, \ldots, p_i\}$, 
	and $\calP(p_i, \ldots, p_{n-1})$ is a tight orientably-regular polytope of type $\{p_i, \ldots, p_{n-1}\}$,
	then $\calP(p_1, \ldots, p_{n-1})$ is a tight orientably-regular polytope of type $\{p_1, \ldots, p_{n-1}\}$.
	\end{theorem}

	\begin{proof}
	Let $\calP_1 = \calP(p_1, \ldots, p_i)$ and $\calP_2 = \calP(p_i, \ldots, p_{n-1})$, 
	which by hypothesis are tight orientably-regular polytopes of the appropriate types. 
	Since $p_{i-1}$ and $p_{i+1}$ are even, \pref{fap} tells us that $\calP_1$ has the FAP 
	with respect to its co-$(i \shm 2)$-faces, and $\calP_2$ has the FAP with respect 
	to its $2$-faces. 
	Then by \tref{fap}, there exists a regular polytope $\calP$ with $(i \shp 1)$-faces 
	isomorphic to $\calP_1$ and co-$(i \shm 2)$-faces isomorphic to $\calP_2$. 
	Moreover, since $\calP_1$ and $\calP_2$ are both tight,
	\pref{tight-recursive} implies that $\calP$ is also tight, and then 
	since $\calP$ is of type $\{p_1, \ldots, p_{n-1}\}$, we find that $|\G(\calP)| = 2p_1 \cdots p_{n-1}$. 
	But also the $(i \shp 1)$-faces of $\calP$ are isomorphic to $\calP_1$, 
	and the co-$(i \shm 2)$-faces are isomorphic to $\calP_2$, 
	and so the standard generators of $\G(\calP)$ must satisfy all 
	the relations of $\G(p_1, \ldots, p_{n-1})$.  
	In particular, $\G(\calP)$ is a quotient of $\G(p_1, \ldots, p_{n-1})$, 
	and so $|\G(p_1, \ldots, p_{n-1})| \ge 2p_1 \cdots p_{n-1}$. 
	On the other hand, $|\G(p_1, \ldots, p_{n-1})| \leq 2p_1 \cdots p_{n-1}$ 
	by \lref{tight-gp}.   Thus $|\G(p_1, \ldots, p_{n-1})| = 2p_1 \cdots p_{n-1}$, 
	and hence also $\G(\calP) \cong \G(p_1, \ldots, p_{n-1})$, 
	and $\calP \cong \calP(p_1, \ldots, p_{n-1})$. 
	Thus $\calP(p_1, \ldots, p_{n-1})$ is a tight polytope of type $\{p_1, \ldots, p_{n-1}\}$, 
	and finally, since all the defining relations of $\G(p_1, \ldots, p_{n-1})$ 
	have even length, we find that $\calP(p_1, \ldots, p_{n-1})$ is orientably-regular.
	\end{proof}

      	Note that the above theorem helps us deal with a large number of possibilities, 
	once we have enough `building blocks' in place. 
	We are assuming that the $(n \shm 1)$-tuple $(p_1, \ldots, p_{n-1})$ is admissible, 
	so that $p_{i-1}$ and $p_{i+1}$ are even divisors of $2p_i$ whenever $p_i$ is odd. 
	Now suppose that $n \ge 6$.  
	If $p_2$ and $p_4$ are both even, then \tref{tight-existence} will apply, 
	and if not, then one of them is odd, say $p_2$, in which case  $p_1$ and $p_3$ 
	must both be even, and again \tref{tight-existence} will apply. 
	Hence this leaves us with just a few cases to verify, namely admissible $(n \shm 1)$-tuples 
	$(p_1, \ldots, p_{n-1})$ with $n = 4$ or $5$ for which there is no $i$ such that 
	$p_{i-1}$ and $p_{i+1}$ are both even. 
	
	The only such cases are as follows: \\[-22 pt] 
	\begin{enumerate}
	\item[$\bullet$] $n = 4$, with $p_1$ odd, $p_2$ even and $p_3$ even, 
	  or dually,  $p_1$ even, $p_2$ even and $p_3$ odd, \\[-22 pt] 
	\item[$\bullet$] $n = 4$, with $p_1$ odd, $p_2$ even and $p_3$ odd,  \\[-22 pt] 
	\item[$\bullet$] $n = 5$, with $p_1$ odd, $p_2$ even, $p_3$ even and $p_4$ odd.
	\end{enumerate}

	We start with the cases where $n = 4$:
	
	\begin{proposition}
	\label{prop:odd-even}
	If $p_1$ is odd, $p_2$ is an even divisor of $2p_1$, and $p_3 \geq 2$, 
	then $\G(p_1, p_2, p_3)$ is a string C-group, 
	and $\calP(p_1, p_2, p_3)$ is a tight orientably-regular polytope of type $\{p_1, q, p_3\}$, 
	for some even $q$ dividing $p_2$. 
	\end{proposition}
	
	\begin{proof}
	Let $\G = \G(p_1, p_2, p_3) = \langle x_0,x_1,x_2,x_3 \rangle$, 
	and let $\overline{\G} = \G(p_1, 2, p_3) = \langle y_0, y_1, y_2, y_3 \rangle$, 
	where $y_i = \overline{x_i}$ is the image of $x_i$ under the natural epimorphism 
	$\pi: \G(p_1, p_2, p_3) \to \G(p_1, 2, p_3)$, for $0 \le i \le 3$. 
	By \pref{cox-2}, we know that $\G(p_1, 2, p_3) \cong [p_1, 2, p_3]$. 
	Also the subgroup $\langle x_0, x_1, x_2 \rangle$ of $\Gamma$ covers $[p_1, 2]$, 
	and this cover is one-to-one on $\langle x_0, x_1 \rangle$, 
	so $\langle x_0, x_1, x_2 \rangle$ is a string C-group, 
	by \pref{quo-crit}.
	A similar argument shows that $\langle x_1, x_2, x_3 \rangle$ is a string C-group.
	Now the intersection of these two string C-groups is $\langle x_1, x_2 \rangle$, since 
	the intersection of their images in $\overline{\G}$ is 
	$ \langle y_0, y_1, y_2 \rangle \cap \langle y_1, y_2,y_3 \rangle = \langle y_1, y_2 \rangle$, 
	and the kernel of $\pi$ is $\langle (x_1 x_2)^2 \rangle$.
	Hence by \pref{facet-vfig}, $\G(p_1, p_2, p_3)$ is a string C-group, and the rest 
	follows easily from \lref{tight-gp}.
	\end{proof}

	\begin{lemma}
	\label{lem:subgroup-2}
	If $p_i = 2$ for some $i$,  
	then in the group $\G(p_1, \ldots, p_{n-1}) = \langle x_0, \ldots, x_{n-1} \rangle$, 
	we have $\langle x_0, \ldots, x_i \rangle \cong \G(p_1, \ldots, p_i)$ 
	and $\langle x_{i-1}, \ldots, x_{n-1} \rangle \cong \G(p_i, \ldots, p_{n-1})$.  
	\end{lemma}
	
	\begin{proof}
	First consider $\Lambda = \langle x_0, \ldots, x_i \rangle$. 
	This is obtainable as a quotient of $\G(p_1, \ldots, p_i)$ by adding extra relations.  
	But also it can be obtained from $\G(p_1, \ldots, p_{n-1})$ by killing the unwanted 
	generators $x_{i+1}, \dots , x_{n-1}$.  
	(Note that for $i \shp 1 \le j \le n \shm 1$, the relation $r_j = 1$ and all relations of 
	the form $(\rho_j \rho_k)^m = 1$ become redundant and may be removed, 
	and the same holds for the relations $r_i = 1$ and $r_{i-1} = 1$ since the 
	assumption that $p_i = 2$ implies that $[x_{i-1},x_i] = (x_{i-1},x_i)^2 =1$ and 
	hence that $x_0, \ldots, x_{i-1}$ commute with $x_i, \ldots, x_{n-1}$.)  
	It follows that $\Lambda \cong \G(p_1, \ldots, p_1)$. 
	Also $\langle x_{i-1}, \ldots, x_{n-1} \rangle \cong \G(p_i, \ldots, p_{n-1})$, 
	by the dual argument.
	\end{proof}
	
	\begin{theorem}
	If $p_1$ is odd, $p_2$ is an even divisor of $2p_1$, and $p_3$ is even, 
	then $\calP(p_1, p_2, p_3)$ is a tight orientably-regular polytope of type $\{p_1, p_2, p_3\}$.
	\end{theorem}
	
	\begin{proof}
	First, the group $\G(p_1, p_2, p_3) = \langle x_0,x_1,x_2,x_3 \rangle$ 
	covers $\G(p_1, p_2, 2) = \langle y_0, y_1, y_2, y_3 \rangle$, say, since $p_3$ is even. 
	Also by \lref{subgroup-2} we know that $\langle y_0,y_1,y_2 \rangle$ is isomorphic 
	to $\G(p_1, p_2)$, and hence in particular, $y_1 y_2$ has order $p_2$. 
	It follows that the order of $x_1 x_2$ is also $p_2$, and then the conclusion follows 
	from \pref{odd-even}. 
	\end{proof}
	
	\begin{theorem}
	If $p_1$ and $p_3$ are odd, and $p_2$ is an even divisor of both $2p_1$ and $2p_3$, 
	then $\calP(p_1, p_2, p_3)$ is a tight orientably-regular polytope of type $\{p_1, p_2, p_3\}$.
	\end{theorem}
	
	\begin{proof}
	Under the given assumptions, the group $\G(p_1,p_2,p_3)$ is obtained from the 
	string Coxeter group $[p_1,p_2,p_3] = \langle x_0,x_1,x_2,x_3 \rangle$ by adding 
	the two extra relations $(x_0 x_1 x_2  x_1 x_2)^2 = 1$ and $(x_3 x_2  x_1 x_2 x_1)^2 = 1$. 
	These imply that the element $\omega = (x_1 x_2)^2$ is inverted under 
	conjugation by each of $x_0$ and $x_3$, and hence by all the $x_i$. 
	Now let $y_i = x_{i-1} x_i$ for $1 \le i \le 3$.  These elements generate the 
	orientation-preserving subgroup $\G^+(p_1, p_2, p_3)$, and they all centralise 
	$\omega = y_2^{\, 2}$.  It follows that $\G^+(p_1, p_2, p_3)$ has presentation 
	\[ \gp{y_1, y_2, y_3}{y_1^{\, p_1}, y_2^{\, p_2}, y_3^{\, p_3}, (y_1 y_2)^2, (y_2 y_3)^2,
	(y_1 y_2 y_3)^2, [y_1, y_2^{\, 2}], [y_3, y_2^{\, 2}]}. \]
	
	We now exhibit a permutation representation of this group on the set $\Z_{p_1} \times \Z_{p_2}$, 
	by letting each $y_i$ induce the permutation $\pi_i$, where 
	\[
	(j,k)^{\pi_2} = (j,k \shp 1) \hbox{  for all } (j,k), 
	\]
	\[
	(j,k)^{\pi_1} = \begin{cases} 
	  (j \shp 1,k) & \hbox{for } k \hbox{ even } \\  
	  (j \shm 1,k \shm 2) & \hbox{for } k \hbox{ odd, } \end{cases}
	  \qquad 
	(j,k)^{\pi_3} = \begin{cases} 
	  (j,k \shm 2j) & \hbox{for } k \hbox{ even } \\  
	  (j,k \shp 2(j \shm 1)) & \hbox{for } k \hbox{ odd. } \end{cases}
	\]
	It is easy to see that $\pi_1$ and $\pi_2$ have orders $p_1$ and $p_2$ respectively
	(since $p_2$ divides $2p_1$), and that the order of $\pi_3$ divides $p_2/2$ 
	and hence divides $p_3$. 
	It is also easy to verify that they satisfy the other defining relations 
	for $\G^+(p_1, p_2, p_3)$, and thus we do have a permutation representation. 
	
	In particular, since $\pi_2$ has order $p_2$, so does $y_2 = x_1 x_2$, 
	and again the conclusion follows from \pref{odd-even}. 
	\end{proof}

	We now handle the remaining case.
	
	\begin{theorem}
	\label{prop:oeeo}
	If $p_1$ and $p_4$ are odd, $p_2$ is an even divisor of $p_1$, and $p_3$ an even divisor of $p_4$, 
	then $\calP(p_1, p_2, p_3, p_4)$ is a tight orientably-regular polytope of type $\{p_1, \ldots, p_4\}$.
	\end{theorem}

	\begin{proof}
	Take $\G = \G(p_1, \ldots, p_4) = \langle x_0, \ldots, x_4 \rangle$, 
	and $\Lambda = \G(p_1, p_2, 2, p_4) = \langle y_0, \ldots, y_4 \rangle$. 
	Then $\G$ covers $\Lambda$, and this induces a cover
	from $\langle x_0, \ldots, x_3 \rangle$ to $\langle y_0, \ldots, y_3 \rangle$, 
	which is isomorphic to $\G(p_1, p_2, 2)$ by \lref{subgroup-2}. 
	Similarly we have a cover from $\langle x_0, x_1, x_ 2\rangle$ to $\langle y_0, y_1, y_2 \rangle$, 
	which is isomorphic to $\G(p_1, p_2)$. 
	But on the other hand, $\langle x_0, x_1, x_2 \rangle$ is a quotient of $\G(p_1, p_2)$, 
	and hence these two groups are isomorphic. 
	In particular, the cover from $\langle x_0, \ldots, x_3 \rangle$ to $\G(p_1, p_2, 2)$ 
	is one-to-one on the facets, so $\langle x_0, \ldots, x_3 \rangle$ is a string C-group. 
	By a dual argument, $\langle x_1, \ldots, x_4 \rangle$ is also a string C-group.
	
	Next, let $\Delta = \G(p_1, 2, 2, p_4) = \langle z_0, \ldots, z_4 \rangle$, 
	and let $\pi$ be the covering homomorphism from $\Gamma$ to $\Delta$. 
	The kernel of $\pi$ is the subgroup generated by $(x_1 x_2)^2$ and $(x_2 x_3)^2$,
	since the defining relations for $\G = \G(p_1, \ldots, p_4)$ imply that these two 
	elements are centralised or inverted under conjugation by each generator $x_i$. 
	In particular, $\ker\pi \subseteq \langle x_1,x_2,x_3 \rangle$. 
	As also the intersection of  $\langle z_0, \ldots, z_3 \rangle$ and  $\langle z_1, \ldots, z_4 \rangle$ 
	in $\Delta$ is $\langle z_1, z_2, z_3 \rangle$, 
	it follows that intersection of  $\langle x_0, \ldots, x_3 \rangle$ and  $\langle x_1, \ldots, x_4 \rangle$
	is $\langle x_1, x_2, x_3 \rangle$. 
	Hence $\G$ is a string C-group.

	Now $\langle x_0, x_1, x_2 \rangle \cong \G(p_1, p_2)$, 
	and by \tref{q-div-2p} we know the polytope $\calP(p_1, p_2)$ has type $\{p_1, p_2\}$. 
	Similarly $\langle x_2, x_3, x_4 \rangle \cong \G(p_3, p_4)$, and $\calP(p_3, p_4)$ has type
	$\{p_3, p_4\}$.
	It follows that $\calP(p_1, p_2, p_3, p_4)$ is an orientably-regular polytope of type $\{p_1, \ldots, p_4\}$.
	In particular, the order of $x_{i-1} x_i$ is $p_i$ (for $1 \le i \le 4$),  
	and so by \lref{tight-gp}, $\calP(p_1, p_2, p_3, p_4)$ is tight.
	\end{proof}

	This gives us all the building blocks we need.  With the help of \tref{tight-existence}, 
	we now know that $\calP(p_1, \ldots, p_{n-1})$ is a tight orientably-regular polytope 
	of type $\{p_1, \ldots, p_{n-1}\}$ whenever $(p_1, \ldots, p_{n-1})$ is admissible, 
	and the proof of \tref{tight-existence2} is complete.

\section{Tight non-orientably-regular polytopes}

	We have not yet been able to completely characterise the Schl\"afli symbols of tight, 
	non-orientably-regular polytopes, but we have made some partial progress. 
	For example, we can easily find an infinite family of tight, non-orientably-regular polyhedra.
	
	\begin{theorem}
	\label{thm:3k-4}
	For every odd positive integer $k$, there exists a non-orientably-regular tight
	polyhedron of type $\{3k, 4\}$, with automorphism group
	\[ \Lambda(k) = \langle\, \rho_0, \rho_1, \rho_2 \mid \rho_0^{\,2}, \rho_1^{\,2}, \rho_2^{\,2}, 
	(\rho_0 \rho_1)^{3k}, (\rho_0 \rho_2)^2, (\rho_1 \rho_2)^4,
	\rho_0 \rho_1 \rho_2 \rho_1 \rho_0 \rho_1 \rho_2 \rho_1 \rho_2  \, \rangle. \]
	\end{theorem}

	\begin{proof}
	We note that $\Lambda(1)$ is the automorphism group of the hemi-octahedron 
	(of type $\{3,4\}$), and that $\Lambda(k)$ covers $\Lambda(1)$, for every $k$. 
	Hence in each $\Lambda(k)$, the order of $\rho_1 \rho_2$ is $4$ (and not $1$ or $2$). 
	Next, because the covering is one-to-one on the vertex-figures, 
	it follows from \pref{quo-crit} that $\Lambda(k)$ is a string C-group.
	Also $\Lambda(k)$ has a relation of odd length, 
	and so it must be the automorphism group of a non-orientably-regular polyhedron.
	
	Now let $N$ be the subgroup generated by the involutions $\rho_2$ and $\rho_1 \rho_2 \rho_1$. 
	Since their product has order $2$, this is a Klein $4$-group.  Moreover, $N$ is 
	normalised by $\rho_1$ and $\rho_2$, and also by $\rho_0$ since $\rho_0$ centralises 
	$\rho_2$ and the final relation in the definition of $\Lambda(k)$ gives 
	$(\rho_1 \rho_2 \rho_1)^{\rho_0} = \rho_1 \rho_2 \rho_1 \rho_2$. 
	It is now easy to see that $N$ is the normal closure of $\langle \rho_2 \rangle$. 
	The quotient $\Lambda(k)/N$ is isomorphic to 
	$\langle\, \rho_0, \rho_1 \, | \, \rho_0^{\,2}, \rho_1^{\,2}, (\rho_0 \rho_1)^{3k}  \, \rangle$, 
	with the final relator for $\Lambda(k)$ becoming trivial, 
	and so $\Lambda(k)/$ is dihedral of order $6k$. 
	In particular, this shows that $\rho_0 \rho_1$ has order $3k$ (and 
	that $|\Lambda(k)| = |\Lambda(k)/N||N| = 24k$). 
	\end{proof}

	The computational data that we have on polytopes with up to 2000 flags (obtained with the 
	help of {\sc Magma} \cite{magma}) suggests that these polyhedra are
	the only tight non-orientably-regular polyhedra of type $\{p, q\}$ with $p$ odd.
	
	Using \tref{3k-4}, it is possible to build tight, non-orientably-regular polytopes in much the
	same as we did in \tref{tight-existence}. In particular, the regular polytope with automorphism
	group $\Lambda(k)$ has the FAP with respect to its 2-faces, and its dual has the FAP with respect
	to its co-0-faces. Then by \tref{fap}, we know there are tight, non-orientably-regular
	polytopes of type $\{4, 3k, 4\}$ for each odd $k$, and of type $\{4, 3k, r\}$ for each odd
	$k$ and each even $r$ dividing $3k$. It is possible to continue in this fashion,
	building tight non-orientably-regular polytopes of every rank.
	
	Finally, just as with orientably-regular polytopes, there are some kinds of Schl\"afli symbol 
	for which no examples be constructed using \tref{fap}. 
	In fact (and in contrast with the situation for orientably-regular polytopes), 
	there seem to be no tight non-orientably-regular polytopes of some of these types at all. 
	For example, there are no tight regular polytopes of type $\{3, 4, r\}$ with $r \geq 3$ 
	and with 2000 flags or fewer, but the reason for this is not clear.

\bibliographystyle{amsplain}
\bibliography{gabe}

\end{document}